\documentclass{amsart}

\usepackage{amssymb}
\usepackage{amsmath}
\usepackage{amsthm}
\usepackage[dvips]{graphicx}
\usepackage{color}

 \newtheorem{theo}{Theorem}[section]
 \newtheorem{lemm}[theo]{Lemma}

\newtheorem{prop}[theo]{Proposition}
\newtheorem{rema}[theo]{Remark}

\newcommand\bR{{\mathbb R}}

\newcommand\bS{{\mathbb S}}

\newcommand\NN{{\mathbb N}}
\newcommand\RR{{{\mathbb R}}}

\def\SS {\mathbb{S}}

\def\la{\langle}
\def\ra{\rangle}

\let\a=\alpha
\let\b=\beta

\let\ve=\varepsilon
\let\lam=\lambda
\let\vp=\varphi
\let\p=\psi
\let\o=\omega
\let\s=\sigma

\let\a=\alpha
\let\b=\beta
\let\g=\gamma

\let\ve=\varepsilon
\let\lam=\lambda

\let\p=\psi
\let\o=\omega

\let\dis=\displaystyle

\newcommand\cF{{\mathcal F}}

\newcommand\cS{{\mathcal S}}
\newcommand\cM{{\mathcal M}}

\newcommand\cK{{\mathcal K}}
\newcommand\cG{{\mathcal G}}
\newcommand\cX{{\mathcal X}}

\def\vb{\mathbf b}
\def\ve{\mathbf e}

\date{today}
\begin{document}
\title[New characterization of Infinite Energy Solutions]
{A New characterization and global regularity \\  of  infinite energy solutions\\  to the homogeneous Boltzmann equation  
}
\author{Yoshinori Morimoto }
\address{Yoshinori Morimoto, Graduate School of Human and Environmental Studies,
Kyoto University,
Kyoto, 606-8501, Japan} \email{morimoto@math.h.kyoto-u.ac.jp}
\author{Shuaikun Wang}
\address{Shuaikun Wang , Department of mathematics, City University of Hong Kong,
Hong Kong, P. R. China
} \email{shuaiwang4-c@my.cityu.edu.hk}
\author{Tong Yang}
\address{Tong Yang, Department of mathematics, City University of Hong Kong,
Hong Kong, P. R. China; \& Department of Mathematics, Shanghai Jiao Tong University, Shanghai, P. R. China
} \email{matyang@cityu.edu.hk}



\date{}

\begin{abstract}
The purpose of this paper is to introduce a new characterization of the characteristic functions
for the study on the measure valued solution to the homogeneous Boltzmann equation so that it precisely captures
the  moment
constraint in physics. This significantly improves the previous result  by Cannone-Karch [CPAM 63(2010), 747-778] 
in the sense that the new characterization gives a complete description
of infinite energy solutions for the  Maxwellian cross section. In addition,  the  global in time smoothing effect of  
the infinite energy solution except for a single Dirac mass initial datum is justified
as for the finite energy solution. 
\end{abstract}
\maketitle

\section{Introduction}\label{s1}

Consider the spatially homogeneous Boltzmann equation,
\begin{equation}\label{bol}
\partial_t f(t,v) 
=Q(f, f)(t,v),
\end{equation}
where $f(t,v)$ is the density distribution of particles with 
velocity $v \in \RR^3$ at time $t$. The most interesting and
important part of this equation is the collision operator
given on the
 right hand side that captures the change rates of the density distribution
through elastic binary collisions:
\[
Q(g, f)(v)=\int_{\RR^3}\int_{\mathbb S^{2}}B\left({v-v_*},\sigma
\right)
 \left\{g(v'_*) f(v')-g(v_*)f(v)\right\}d\sigma dv_*\,,
\]
where for $\sigma \in \SS^2$
$$
v'=\frac{v+v_*}{2}+\frac{|v-v_*|}{2}\sigma,\,\,\, v'_*
=\frac{v+v_*}{2}-\frac{|v-v_*|}{2}\sigma,\,
$$
that follow from the conservation of momentum and energy,
\[ v' + v_*' = v+ v_*, \enskip |v'|^2 + |v_*'|^2 = |v|^2 + |v_*|^2.
\]
In this paper, we consider
the Cauchy problem of \eqref{bol} with initial datum
\begin{equation}\label{initial}
f(0,v) = F_0\ge 0.
\end{equation}

Motivated by some physical models, we assume that the non-negative cross section 
$B$  takes the form
\begin{equation*}
B(|v-v_*|, \cos \theta)=\Phi (|v-v_*|) b(\cos \theta),\,\,\,\,\,
\cos \theta=\frac{v-v_*}{|v-v_*|} \, \cdot\,\sigma\, , \,\,\,
0\leq\theta\leq\frac{\pi}{2},
\end{equation*}
where
\begin{align}
&\Phi(|z|)=\Phi_\gamma(|z|)= |z|^{\gamma}, \enskip \mbox{for some $\gamma>-3$},   \notag \\
& b(\cos \theta)\theta^{2+2s}\ \rightarrow K\ \
 \mbox{when} \ \ \theta\rightarrow 0+,  \enskip 
\mbox{for $0<s<1$ and $K>0$. }\label{1.2}
\end{align}
Throughout this paper, we will only consider the case when
\[
\gamma = 0,\enskip 0<s<1\,,
\]
which is called Maxwellian molecule type cross section, because the analysis
relies on the simpler form of the equation after taking Fourier transform
in $v$ by the Bobylev formula. 
As usual, the range of $\theta$ can be restricted to $[0,\pi/2]$, by replacing $b(\cos\theta)$
 by its ``symmetrized'' version  
\[
[ b(\cos \theta)+b(\cos (\pi-\theta))]{\bf 1}_{0\le \theta\le \pi/2}.
\]

It is now well known that the angular singularity in the cross section leads
to the gain of regularity in the solution. The purpose of this paper is to
show that this still holds for measure valued solutions. Moreover, we 
improve the previous work on existence theory by Cannonne-Karch \cite{Cannone-Karch-1} because we introduce a new classification of the characteristic functions
so that the moment constraints can be precisely captured in the Fourier
space.

We start with the following a slightly general assumption on the cross section
\begin{align}\label{index-sing}
\exists \alpha_0 \in (0, 2]\enskip\mbox{ such that} \enskip  (\sin \theta/2)^{\alpha_0} b(\cos \theta) 
\sin \theta  \in L^1((0, \pi/2]),
\end{align}
which is fulfilled for $b$ with \eqref{1.2} if $2s < \alpha_0$.

Denote by $P_\alpha(\RR^d)$, $0 \le \alpha \le 2$ the  probability measures  $F$ on $\RR^d$, $d \ge 1$, such that
\[
\int_{\RR^d} |v|^\alpha dF(v) < \infty, \,
\] 
and moreover when {$\a \ge 1$ }, it requires that 
\begin{align}\label{mean}
\int_{\RR^d} v_j dF(v) = 0, \enskip j =1,\cdots d\,.
\end{align}
Following Cannone-Karch \cite{Cannone-Karch-1},  the Fourier transform of a probability measure $F \in P_0(\RR^d)$ called a characteristic function
is defined by
by 
\[ 
\varphi(\xi) = \hat f(\xi) = \cF(F)(\xi) =\int_{\RR^d} e^{-iv\cdot \xi} dF(v).
\]
Put 
$\cK = \cF(P_0(\RR^d))$.
Inspired by a series of works by Toscani and co-authors\cite{Carlen-Gabetta-Toscani, Gabetta-Toscani-Wennberg, toscani-villani}, Cannone-Karch defined a subspace 
$\cK^\a$ for $\a \ge 0$  as follows:
\begin{align}\label{K-al}
\cK^\alpha =\{ \varphi \in \cK\,;\, \|\varphi - 1\|_{\alpha} < \infty\}\,,
\end{align}
where 
\begin{align}\label{dis-norm}
\|\varphi - 1\|_{\alpha} = \sup_{\xi \in \RR^d} \frac{|\varphi(\xi) -1|}{|\xi|^\alpha}.
\end{align}
The space $\cK^\alpha $ endowed with the distance 
\begin{align}\label{distance}
\|\varphi - \tilde \varphi\|_{\alpha} = \sup_{\xi \in \RR^d} \frac{|\varphi(\xi) -
\tilde \varphi(\xi) |}{|\xi|^\alpha}
\end{align}
is a complete metric space (see Proposition 3.10 of \cite{Cannone-Karch-1}).
It follows that $\cK^\alpha =\{1\}$ for all $\alpha >2$ and the 
following embeddings
(Lemma 3.12 of \cite{Cannone-Karch-1}) hold
\[
\{1\} \subset \cK^\alpha \subset \cK^\beta \subset \cK^0 =\cK \enskip \enskip
\mbox{for all $2\ge \alpha \ge \beta \ge 0$}\,\,.
\]
With this classification on the characteristic functions, the global
existence of solution in $\cK^\alpha$ was studied in \cite{Cannone-Karch-1}(see also \cite{morimoto-12}).
However, 
even though the inclusion $\cF(P_\a(\RR^d)) \subset \cK^\a$ holds (see Lemma 3.15 of \cite{Cannone-Karch-1}), 
the space $\cK^\a$ is strictly bigger than $\cF(P_\a(\RR^d))$ for $\a \in (0,2)$, in other word, $\cF^{-1}(\cK^\a) \supsetneq P^\a(\RR^d)$.  Indeed,  it is shown (see Remark 3.16 of \cite{Cannone-Karch-1}) that the function 
$\varphi_\a (\xi) = e^{-|\xi|^\a}$, with $\a \in (0,2)$, belongs to $\cK^\a$ but $p_\a(v)= \cF^{-1} (\varphi_\a)(v)$, the density of $\a-$stable symmetric L\'evy process, 
is not contained in $P_\a(\RR^d)$. 
Hence, the solution obtained in the function space $\cK^\alpha$ does not
represent the moment properties in physics even when it is assumed initially.

In order to fill this gap, one main purpose of
this paper is to introduce a new classification on the characteristic functions
as follows. Set 
\begin{align}\label{M-al}
\cM^\alpha =\{ \varphi \in \cK\,;\, \|\varphi - 1\|_{\cM^\alpha} < \infty\}\,,\enskip \a \in (0,2)\,,
\end{align}
where 
\begin{align}\label{integral-norm}
\|\varphi -1\|_{\cM^\a}  = \int_{\RR^d} \frac{|\varphi(\xi)-1 |}{|\xi|^{d+\a} }d\xi\,.
\end{align}
For $\varphi, \tilde \varphi \in \cM^\a$, put 
\begin{align*}
\|\varphi -\tilde \varphi \|_{\cM^\a}  = \int_{\RR^d} \frac{|\varphi(\xi)-\tilde \varphi(\xi) |}{|\xi|^{d+\a} }d\xi\,,
\end{align*}
and, for any $\b \in (0,\a]$, we introduce the distance 
\begin{align}\label{M-dist}
{dis}_{\a,\b} (\varphi ,\tilde \varphi )= \|\varphi - \tilde \varphi\|_{\cM^\a}   + \|\varphi - \tilde \varphi\|_\b\,.
\end{align}

With the above notations, 
 the first main result in this paper can be stated as follows.

\begin{theo}\label{complete-1}  If $0 < \b \le \alpha < 2$, then    
the space $\cM^\a$ is a complete metric space endowed with the distance $
dis_{\a,\b} (\varphi ,\tilde \varphi )$. If $\b, \b'$ are  in $(0,\a]$,  both distances  $dis_{\a,\b} (\cdot, \cdot)$ and $dis_{\a,\b'} (\cdot, \cdot)$ are equivalent
in the following sense:
\begin{align}\label{fourier-equiv}
\mbox{$\forall \varepsilon >0$,  $\exists \delta>0$ s.t. $dis_{\a,\b} (\varphi ,\varphi_0 ) < \delta$  $\Rightarrow$ $dis_{\a,\b'} (\varphi , \varphi_0 )< \varepsilon $.}
\end{align}
Moreover, we have
\begin{align}\label{trivial-inclusion}
&\cK^\b \subset \cM^\a \enskip \mbox{if $\a < \b$ and $\a \in (0,2)$,}\\
&\label{inclusion-spa}
\cM^\a \subset \cF(P_\a(\RR^d)) \, \big (\, \,  \subsetneq \cK^\a \,\, \big) \enskip \mbox{for} \enskip \a \in (0,2)\,,\\
&\label{equivalence-space}
\cM^\a = \cF(P_\a(\RR^d))\,, \enskip \mbox{furthermore if  $\a \ne 1$ },
\end{align}
and the fact that $\displaystyle \lim_{n \rightarrow \infty} dis_{\a,\b}(\varphi_n, \varphi) = 0$, for $\varphi_n, 
\varphi 
 \in \cM^\a$,
implies
\begin{align}\label{measure-convergence}
&\lim_{n \rightarrow \infty} \int \psi(v) dF_n(v) = \int \psi(v) dF(v)  
\mbox{ for any  $\psi \in C(\RR^d)$} \\
&\mbox{\qquad  \qquad \qquad 
 satisfying 
the growth condition $|\psi(v)| \lesssim \la v \ra^\a$,}\notag
\end{align}
where $F_n = \cF^{-1}(\varphi_n), F = \cF^{-1}(\varphi) \in P_\a(\RR^d)$.
\end{theo}
\begin{rema}\label{remark-additional}
If $F \in P_1(\RR^d)$  satisfies 
$
\int_{\RR^d} \la v \ra \log \la v \ra dF_0(v) < \infty,
$
then we have $\cF(F) \in \cM^1$ (see Proposition \ref{chara-1} below).
\end{rema}
{\begin{rema}\label{Wasserstein}
The weak convergence \eqref{measure-convergence} is equivalent to the one
given by the Wasserstein distance (see Theorem 7.12 of \cite{villani3}). 
\end{rema}}

Thanks to the  new characterization of $P_\a$ by its exact Fourier image $\cM^\a$, we can improve results, given in 
\cite{Cannone-Karch-1, morimoto-12, MY},  concerning the existence and the smoothing effect of measure valued
solutions to the Cauchy problem for the 
spatially homogeneous Boltzmann equation of Maxwellian molecule type cross section without angular cutoff that will be stated as follows.

\begin{theo}\label{existence-base-space}
Assume that $b$ satisfies \eqref{index-sing} for some $\a_0 \in (0,2)$ and let  $\a \in [\a_0, 2), \a \ne 1$. 
If $F_0 \in P_\a(\RR^3)$, then there exists a unique measure valued solution $F_t \in C([0,\infty), P_\a(\RR^3))$ to the Cauchy problem 
\eqref{bol}-\eqref{initial}, where the continuity with respect to $t$ is according to  the topology deduced in the sense of \eqref{measure-convergence}.
\end{theo}
\begin{rema}\label{remark-1}
When $F_0 \in P_1(\RR^3)$, the result is slightly weaker, that is, for any $T >0$ there exists a constant $C_T >0$ such that
\begin{equation}\label{1-mome}
\sup_{0 \le t \le T} \int_{\RR^3} \la v \ra dF_t(v) \le C_T.
\end{equation}
Moreover,  if $F_0 \in P_1(\RR^3)$ satisfies 
$
 \int_{\RR^3} \la v \ra \log \la v \ra dF_0(v) < \infty,
$
then  the solution $F_t \in C([0,\infty), P_1(\RR^3))$. 
\end{rema}

The proof of the above theorem is given in the Fourier space. In fact, 
by letting $\varphi(t,\xi) = \int_{\RR^3}e^{-iv\cdot \xi} dF_t(v)$ and $\varphi_0= \cF(F_0)$, 
it follows from the Bobylev formula that
the Cauchy problem \eqref{bol}-\eqref{initial} is reduced to 
\begin{equation}\label{c-p-fourier}
\left \{ 
\begin{array}{l}\dis \partial_t \varphi(t,\xi)
=\int_{\SS^2}b\left(\frac{\xi \cdot \sigma}{|\xi|}\right) \Big( \varphi(t,\xi^+)\varphi(t, \xi^-) - \varphi(t, \xi)
\varphi(t,0)\Big) d\sigma, \\\\
\dis \varphi(0,\xi)=\varphi_0(\xi), \enskip \mbox{where} \enskip
\dis \xi^\pm = \frac{\xi}{2} \pm \frac{|\xi|}{2} \sigma\,.
\end{array}
\right.
\end{equation}
By  Theorem \ref{complete-1}, to prove Theorem \ref{existence-base-space}
it suffices to show 
\begin{theo}\label{fourier-space}
Assume that $b$ satisfies \eqref{index-sing} for some $\a_0 \in (0,2)$ and let  $\a \in [\a_0, 2)$. 
If the initial datum $\varphi_0$ belongs to $\cM^\a$, then there exists a unique classical solution $\varphi(t,\xi) \in C([0,\infty), \cM^\a)$
to the Cauchy problem \eqref{c-p-fourier} such that
\begin{align}\label{fourier-continuity}
dis_{\a,\a} (\varphi(t,\cdot), \varphi(s,\cdot))\lesssim |t-s| e^{\lambda_\a \max \{t,s\}}dis_{\a,\a} (\varphi_0, 1)\,.
\end{align}
Here 
\begin{equation}\label{05}
  \lam_\a= 2 \pi \int_0^{\pi/2} b\left( \cos \theta \right)
             \left(\cos^\a\frac\theta2+\sin^\a\frac\theta2-1\right) \sin \theta d \theta >0 \,. 
\end{equation}
Furthermore, if $\vp(t,\xi),\tilde \vp(t,\xi)\in C([0,\infty),\cM^\a)$ are two solutions to the Cauchy problem \eqref{c-p-fourier} 
with initial data $\vp_0,\tilde \vp_0\in\cM^\a$, respectively, then for any $t>0$, the following two stability estimates hold
\begin{align}\label{04}
 & ||\vp(t)-\tilde \vp(t)||_{\cM^\a} \leq e^{\lam_\a t}||\vp_0-\tilde \vp_0||_{\cM^\a}\,,\\
 &||\vp(t)-\tilde \vp(t)||_{\a} \leq e^{\lam_\a t}||\vp_0-\tilde \vp_0||_{\a}\,. \label{alpha-stability}
\end{align}
\end{theo}
\begin{rema}\label{rem-1234}
Since $\vp_0,\tilde \vp_0\in\cM^\a \subset \cK^\a$, the stability estimate \eqref{alpha-stability} is nothing but (13) of \cite{morimoto-12}.
\end{rema}

As in \cite{MY}, set  $\tilde P_\a(\RR^3) =\cF^{-1}(\cK^\a)$ endowed with the distance \eqref{distance}.
We are now ready to state the global in time smoothing effect for infinite energy solutions to the Cauchy problem \eqref{bol}-\eqref{initial},
which is an improvement of Theorem 1.3 of \cite{MY}, where the time global smoothing effect for finite energy solutions and 
a short time smoothing effect for infinite energy solutions were proved.

\begin{theo}\label{smoothing-thm} 
Let $b(\cos \theta)$ satisfy \eqref{1.2} 
with $0<s<1$
and let $\alpha \in (2s, 2]$.
If $F_0 \in \tilde P_\alpha(\RR^3)$ is not a single Dirac mass and  $f(t,v)$ is a unique solution 
in $C([0,\infty), \tilde P_\alpha(\RR^3))$ to the Cauchy problem \eqref{bol}-\eqref{initial},
then  $f(t, \cdot )$ belongs to $H^\infty(\RR^3)$ for any $t >0$.  If $F_0 \in  P_\alpha(\RR^3)$ then
$f(t, \cdot )$ belongs to $L^1_\a (\RR^3) \cap H^\infty(\RR^3) $ for any $t >0$, and moreover when $\a \ne 1$, 
we have $f(t,v) \in C((0,\infty), L^1_\a(\RR^3)\cap H^\infty(\RR^3))$. 
\end{theo}

The rest of the paper will be organized as follows. In the following
section, we will prove the needed estimates for the new classification
$\cM^\alpha$ of $\cK$. The existence of measure valued solution in 
the function space $\cM^\alpha$ will be discussed in Section 3 and the
global in time smoothing effect of the solution will be proved
in the last section.

\section{Characterization of $P_\alpha$}\label{s2}
\setcounter{equation}{0}

In this section, we will prove the estimates on $\cM^\alpha$ stated in
Theorem \ref{complete-1}. For this, we first prove the following
two propositions.

\begin{prop}\label{chara-1}
If $0 < \alpha <2$, $\a \ne 1$ and if  $F \in P_\alpha(\RR^d)$,
then there exists a $C >0$ depending only on $d$ and $\a$ such that for $\varphi = \cF(F)$ we have 
\begin{align}\label{fourier-int}
\|\varphi -1\|_{\cM^\a} = \int_{\RR^d} \frac{|\varphi(\xi)-1|}{|\xi|^{d+\alpha}} d \xi \le C \int_{\RR^d} \la v \ra^\a dF(v)  \,.
\end{align}
When $\alpha =1$, we have $\varphi  \in \cM^\a$  if  $F \in P_1(\RR^d)$ satisfies 
\begin{align}\label{moments-log}
\int \la v \ra \log \la v \ra  dF(v) < \infty\,.
\end{align}
\end{prop}
\begin{proof}
Note
\begin{align*}
\varphi(\xi) -1&= \int_{\RR^d} (e^{-iv\cdot \xi}-1) d F(v)\\
& = \int_{\RR^d}\Big(e^{-iv\cdot \xi}- \sum_{\ell \le [\alpha]} \frac{1}{\ell !} (-iv\cdot \xi)^\ell \Big)dF(v),
\end{align*}
because of \eqref{mean} when $\alpha \ge 1$. By the polar coordinate $\xi = \rho \omega$ ($\rho >0, \omega \in  \SS^{d-1}$), we have
\[
\Big|e^{-iv\cdot \xi}- \sum_{\ell \le [\alpha]} \frac{1}{\ell !} (-iv\cdot \xi)^\ell 
\Big| \lesssim \min \big \{(|v|\rho)^{[\alpha]+1}\, ,\,\big (1+(|v|\rho)^{[\alpha]}\big)\big \}\,,
\]
so that 
\begin{align*}
&\int_{\{|\xi|\le 1\}} \frac{|1-\varphi(\xi)|}{|\xi|^{d+\alpha}} d \xi 
\le \int_{\RR^d}\Big(  \int_{{\{|\xi|\le 1\}}}\Big|e^{-iv\cdot \xi}- \sum_{\ell \le [\alpha]} \frac{1}{\ell !} (-iv\cdot \xi)^\ell \Big| \frac{d \xi}{|\xi|^{d+\alpha}} \Big)dF(v)\\
& \lesssim 
\int_{\RR_v^d} \Big( \int_0^{\la v \ra^{-1}}\frac{(|v|\rho)^{[\alpha]+1} }{\rho^{1+ \alpha}}d\rho
+\int_{\la v \ra^{-1}}^1 \frac{1+ (|v|\rho)^{[\alpha]} }{\rho^{1+\alpha}}d\rho\Big)dF(v)\\
& \lesssim \int_{\RR_v^d} \la v \ra^\alpha dF(v)  \enskip \mbox{if $\alpha \ne 1$}, 
\enskip \lesssim \int_{\RR_v^d} \la v \ra \log \la v \ra dF(v)  \enskip \mbox{if $\alpha =1$}. 
\end{align*}
Consequently, we have the  estimate \eqref{fourier-int}  because $\int_{\{|\xi| > 1\}} \frac{1}{|\xi|^{d+\alpha}} d \xi < \infty$.
\end{proof}

Conversely, we have 
\begin{prop}\label{chara-2}
Let $\a \in (0,2)$ and let $\cM^\a$ be a subspace of $\cK = \cF(P_0(\RR^d))$ defined by \eqref{M-al}.
Then we have the inclusion \eqref{inclusion-spa}. 
Furthermore, for $M \in[1, \infty]$, if we put 
\begin{align}\label{uniform-constant}
c_{\alpha,d,M} = \int_{\{|\zeta| \le M\}} \frac{\sin^2 (\ve_1 \cdot \zeta/2)}{|\zeta|^{d+\alpha}} d \zeta >0,
\end{align}
and if $F = \cF^{-1}(\varphi)$ for $\varphi \in \cM^\a$, then for any $R >0$ we have
\begin{align}\label{uniform-moment-est}
\int_{\{|v| \ge R \}} |v|^\alpha dF(v) \le \frac{1}{2c_{\alpha,d,1}}
\int_{\{|\xi| \le 1/R\}} \frac{|1-\varphi(\xi)|}{|\xi|^{d+\alpha}} d \xi.
\end{align}
Moreover,
\begin{align}\label{moment-es}
\int_{\RR^d} |v|^\alpha dF(v) \le \frac{1}{2c_{\alpha,d, \infty}}
\int_{\RR^d} \frac{|1-\varphi(\xi)|}{|\xi|^{d+\alpha}} d \xi  \,.
\end{align}

\end{prop}
\begin{proof}
Since $|\varphi(\xi)| \le \varphi(0) =1$, we have
\begin{align*}
\int_{\{|\xi| \le M/R\}} \frac{|1-\varphi(\xi)|}{|\xi|^{d+\alpha}} d \xi  &\ge 
\int_{\{|\xi| \le M/R\}}\frac{\mbox{Re}~ (1-\varphi(\xi))}{|\xi|^{d+\alpha}} d \xi \\
&= \int_{\RR_v^d} \Big(\int_{\{|\xi| \le M/R\}}\frac{2 \sin^2 (v\cdot \xi/2)}{|\xi|^{d+\alpha}} d \xi \Big)
dF(v).
\end{align*}
By the change of variable $|v| \xi = \zeta$ and by using
the invariance of the rotation, we have
\begin{align*}
\int_{\{|\xi| \le M/R\}} \frac{\sin^2 (v\cdot \xi/2)}{|\xi|^{d+\alpha}} d \xi
&=|v|^\alpha \int_{\{|\zeta| \le M |v|/R\}} \frac{\sin^2 (\ve_1 \cdot \zeta/2)}{|\zeta|^{d+\alpha}} d \zeta \\
& \ge |v|^\alpha {\bf 1}_{\{|v| \ge R\}} c_{\alpha,d,M} \,,
\end{align*} 
which yields \eqref{uniform-moment-est}, with the choice of $M=1$. 
By letting  $M \rightarrow \infty$  and   $R \rightarrow 0$, we obtain  \eqref{moment-es}.  In order to complete the proof of \eqref{inclusion-spa}, it remains to show \eqref{mean} when  $\a \ge 1$. 
Suppose that 
\[
\exists \varphi \in \cM^\a \enskip s.t. \enskip a := \int v d F(v) \ne 0\,, \enskip F = \cF^{-1}(\varphi).
\]
Since $F(\cdot + a)$ belongs to $P_\a(\RR^d)$ and satisfies \eqref{mean}, it follows from  Proposition \ref{chara-1} that its Fourier transform  
$$\varphi_a(\xi)  = \int e^{-iv\cdot \xi} dF(v+a)= e^{ia \cdot \xi} \varphi(\xi)$$ 
belongs to $\cM^\a$,  that is,
\begin{align*}
\infty > \int_{\RR^d} \frac{|1-e^{ia \cdot \xi} \varphi(\xi)|}{|\xi|^{d+\alpha}} d \xi = \int_{\RR^d} \frac{|(e^{-ia \cdot \xi}-1) + (1-\varphi(\xi))|}{|\xi|^{d+\alpha}} d \xi\,.
\end{align*}
Since $\varphi \in \cM^\a$, we obtain $\int_{\RR^d} \frac{|e^{-ia \cdot \xi}-1|}{|\xi|^{d+\alpha}} d \xi< \infty$, which contradicts $a \ne 0$. In
fact, by the rotation, we can assume $a = (0,\cdots,0, |a|)$ and hence 
\[
\infty > \int_{\RR^d} \frac{|e^{-ia \cdot \xi}-1|}{|\xi|^{d+\alpha}} d \xi
\ge |\SS^{d-2}|\int_0^{\pi/2}\Big( \int_0^{1/|a|} \frac{\sin (\rho |a| \cos \theta)}{\rho^{1+\a}} d\rho\Big) d \theta\,.
\]
\end{proof}

We are now ready to prove Theorem \ref{complete-1}.

\begin{proof}[Proof of Theorem \ref{complete-1}]
Suppose that $\{\varphi_n \}_{n=1}^\infty \subset \cM^\a$ satisfies
\[
dis_{\a,\b} (\varphi_n, \varphi_m) \rightarrow 0 \enskip (n,m \rightarrow \infty) \enskip , \enskip 0< \b \le \a <2.
\]
Since it follows from Proposition 3.10 of \cite{Cannone-Karch-1} that $\cK^\b$ is a complete metric space,
we have the limit (pointwise convergence)
$$\varphi(\xi) = \lim_{n \rightarrow \infty } \varphi_n(\xi)  \in \cK^\b \subset \cK.
$$
For any fixed $R >1$ we have 
\[
\int_{\{R^{-1}\le |\xi| \le R\}}\frac{|\varphi_n (\xi) -1|}{|\xi|^{d+\a}} d\xi \le \sup_{n} \|\varphi_n -1 \|_{\cM^\a} < \infty.
\]
Taking the limit with respect to $n$ and letting $R \rightarrow \infty$, we have $\varphi \in \cM^\a$.
Now it is easy to see that  $dis_{\a,\b} (\varphi_n, \varphi) \rightarrow 0$.

To prove \eqref{trivial-inclusion}, we will show that 
there exists a $C_{\a,\b,d} >0$
depending only on $\a, \b, d$ such that
\begin{align}\label{trivial-estimate}
\|\varphi -1 \|_{\cM^\a} \le C_{\a,\b,d} \|\varphi -1\|_\b^{\a/\b}, \enskip \mbox{for $\varphi \in \cK^\b$\,.}
\end{align}
Since $|\varphi(\xi)|\le \varphi(0) =1$, for $R >0$ we have 
\begin{align*}
\|\varphi -1\|_{\cM^a} \leq \int_{\{|\xi| > 1/R\}} \frac{2}{|\xi|^{d+\alpha}} d \xi +\int_{\{|\xi| \le 1/R\}} \frac{\|\varphi-1\|_\b}{|\xi|^{d+\alpha-\b}} d \xi \,,
\end{align*}
which gives \eqref{trivial-estimate} by taking $R \sim \|\varphi -1\|_\b^{1/\b}$.  
By Proposition \ref{chara-2}, we get \eqref{inclusion-spa}, that is, $\cM^\a \subset \cF(P_\a(\RR^d))$.
Since it follows from Lemma 3.15 of \cite{Cannone-Karch-1} that
$\cF(P_\a(\RR^d))\subset \cK^\a$,  
we have $\cM^\a \subset  \cK^\a$. More precisely,  there exist $C_1, C_2 >0$ such that for $\varphi \in \cM^\a,\, F = \cF^{-1}(\varphi)$, we have 
\begin{align}\label{embed}
\|\varphi -1\|_\a \le C_1 \int_{\RR^d} |v|^\a dF(v) \le C_2 \|\varphi -1\|_{\cM^\a}\,,
\end{align}
where the second inequality follows from \eqref{moment-es}.

{We will now show \eqref{fourier-equiv}. Let $\b \in (0,\a)$ and suppose that 
$
dis_{\a,\b} (\varphi, \varphi_0) < \delta
$ for $\varphi, \varphi_0  \in \cM^\a$.
Noting that
\begin{align*}
&\frac{\varphi(\xi) -\vp_0(\xi)}{|\xi|^\a} =\int_{\RR^d} \frac{e^{-i\xi \cdot v} -1}{(|\xi||v|)^\a}|v|^\a \Big(dF(v) -dF_0(v)\Big), \enskip &\mbox{if $\a \in (0,1)$,}\\
&\frac{\varphi(\xi) -\vp_0(\xi)}{|\xi|^\a} =\int_{\RR^d} \frac{e^{-i\xi \cdot v} -1+ i \xi\cdot v}{(|\xi||v|)^\a}|v|^\a \Big(dF(v) -dF_0(v)\Big), \enskip &\mbox{if $\a \in [1,2)$,}
\end{align*}
we have, for any $1 < \tilde R < R$,  
\begin{align*}
\sup_{|\xi| <R^{-1}}\frac{|\varphi(\xi) -\vp_0(\xi)|}{|\xi|^\a}& \le \sup_{|\xi| <R^{-1}}
\int_{\{|v| < \tilde R\}}\Big(|v||\xi|\Big)^{1+[\a]-\a} |v|^\a\Big(dF(v) + dF_0(v)\Big)\\
&\quad + 3 \int_{\{|v| \ge \tilde R\}}|v|^\a\Big(dF(v) + dF_0(v)\Big)\\
&= I_1(R, \tilde R) + I_2(\tilde R)\,.
\end{align*}
It follows from \eqref{uniform-moment-est} that 
\begin{align*}
I_2(\tilde R) &\le \frac{3}{2 c_{\a,d,1}}
\int_{\{|\xi| \le 1/{\tilde R}\}} \frac{|1-\varphi(\xi)| + |1-\varphi_0(\xi)|}{|\xi|^{d+\alpha}} d \xi \\
&\le \frac{3}{c_{\a,d,1}}\Big( \int_{\{|\xi| \le 1/{\tilde R}\}} \frac{ |1-\varphi_0(\xi)|}{|\xi|^{d+\alpha}} d \xi  + \|\vp -\vp_0\|_{\cM^a} \Big).
\end{align*}
Hence,  for any $\varepsilon >0$ there exists a $\tilde R>1$ such that 
$
I_2(\tilde R) < \varepsilon + {3\delta}/{c_{\a,d,1}}
$. If $\tilde R >1$ is fixed as above, then  we have
\begin{align*}
I_1(R, \tilde R)  \le  2 \frac{\tilde R^{1+[\a]} }{R^{1+[\a]-\a}} \rightarrow 0\,, \enskip R \rightarrow \infty\,.
\end{align*}
Consequently, for any $\varepsilon >0$  there exists $R>1$ such that
\[
\sup_{|\xi| <R^{-1}}\frac{|\varphi(\xi) -\vp_0(\xi)|}{|\xi|^\a} \le 2\varepsilon + \frac{3\delta}{c_{\a,d,1}}.
\]
On the  other hand,
\[
 \sup_{ |\xi| \ge  R^{-1} } \frac{|\varphi(\xi) -\varphi_0(\xi)|}{|\xi|^\a} \le  
 R^{\a -\b} \sup_{ |\xi| \ge  R^{-1} } \frac{|\varphi(\xi) -\varphi_0(\xi)|}{|\xi|^\b} \le R^{\a -\b}\delta\,.
\]
Thus, we obtain \eqref{fourier-equiv} with $\b' = \a$.  Other cases are easier
because  for $\b' >0$,
\[
\sup_{|\xi| \ge R} \frac{|\varphi(\xi) -\varphi_0(\xi)|}{|\xi|^{\b'}} \rightarrow 0 \enskip (R \rightarrow \infty)\,.
\]

}

{The exact characterization formula \eqref{equivalence-space}, when 
$\a \ne1$,}
is a direct consequence of 
\eqref{fourier-int} and \eqref{uniform-moment-est}.  Suppose that, for $F_n, F \in P_\a(\RR^d)$,  we have 
$$
\varphi_n = \cF(F_n), \varphi = \cF(F) \in \cM^\a,\enskip \mbox{and} \enskip 
\displaystyle \lim_{n \rightarrow \infty} dis_{\a,\b}(\varphi_n, \varphi) = 0\,.
$$ 
Note that for $R >1$
\[
\int_{\{|\xi| \le 1/R\}} \frac{|1-\varphi_n (\xi)|}{|\xi|^{d+\alpha}} d \xi  
\le \int_{\{|\xi| \le 1/R\}} \frac{|1-\varphi(\xi)|}{|\xi|^{d+\alpha}} d \xi  + \|\varphi_n -\varphi\|_{\cM^\a} \,. 
\]
It follows from \eqref{uniform-moment-est} that for any $\varepsilon >0$ there exist $R >1$ and $N \in \NN$ such that 
 \[  
\int_{\{|v| \ge R \}} |v|^\alpha dF_n(v) +  \int_{\{|v| \ge R \}} |v|^\alpha dF(v) < \varepsilon\, \enskip \mbox{if $n \ge N$}.
\]
This shows \eqref{measure-convergence} because $\varphi_n \rightarrow \varphi$ in $\cS'(\RR^d)$ and so
$F_n \rightarrow F$ in $\cS'(\RR^d)$. And this completes the proof of the
theorem.
\end{proof}

\section{Proofs of Theorem \ref{existence-base-space} and Theorem  \ref{fourier-space}}\label{s-34}

The main purpose of this section concerns with the existence of measure
valued solutions in the new classification of the characteristic functions.
We only need to prove  Theorem  \ref{fourier-space} because Theorem \ref{existence-base-space} follows by using Theorem \ref{complete-1}.

\subsection{Existence under the cutoff assumption} As usual, the
existence for non-cutoff cross section is based on the cutoff approximations.
Hence,
in this subsection, we  assume that 
\begin{align}\label{ang-cutoff}
\gamma_2=\int_{\SS^2} b \Big (\frac{\xi\cdot\sigma}{|\xi|}\Big )  d\s = 2\pi \int_0^{\pi/2} b(\cos \theta) \sin \theta d \theta < \infty \,.
\end{align}

The existence and stability with cutoff cross section can be stated
as follows.

\begin{prop}\label{cutoff-existence}
Let $\a\in (0,2)$ and $b$ satisfy \eqref{ang-cutoff}. For every initial datum $\varphi_0 \in \cM^\a$, there exists 
a unique classical solution to the Cauchy problem \eqref{c-p-fourier} satisfying $\varphi(t) \in C([0, \infty), \cM^\a)$. 
Furthermore, if $\vp(t,\xi),\tilde \vp(t,\xi)\in C([0,\infty),\cM^\a)$ are two solutions to the Cauchy problem \eqref{c-p-fourier} 
with initial data $\vp_0,\tilde \vp_0\in\cM^\a$, respectively, then for any $t>0$ we have 
\begin{align}\label{cutoff-sta}
 & ||\vp(t)-\tilde \vp(t)||_{\cM^\a} \leq e^{\lam_\a t}||\vp_0-\tilde \vp_0||_{\cM^\a}\,,
\end{align}
where $\lambda_\a$ is defined by \eqref{05}.
\end{prop}
Following the subsection 4.2 of \cite{Cannone-Karch-1}, we consider
 the nonlinear operator,
$$
\cG(\vp)(\xi)\equiv\int_{\bS^2} b \Big (\frac{\xi\cdot\sigma}{|\xi|}\Big )\vp(\xi^+)\vp(\xi^-)d\sigma.
$$
Then,  problem \eqref{c-p-fourier} can be formulated by
\begin{align}\label{integral-eq}
\vp(\xi,t)=\vp_0(\xi)e^{-\g_2t}+\int_0^t e^{-\g_2(t-\tau)}\cG(\vp(\cdot,\tau))(\xi)d\tau.
\end{align}

For the nonlinear operator $\cG(\cdot)$, we have the following estimate.

\begin{lemm}\label{lemm01}
Let $\a\in (0,2)$ and $b$ satisfy \eqref{ang-cutoff}. For all $\vp,\tilde \vp\in\cM^\a$, we have
\begin{equation}\label{06}
  ||\cG(\vp)(\xi)-\cG(\tilde \vp)(\xi)||_{\cM^\a}\leq\g_\a||\vp-\tilde \vp||_{\cM^\a},
\end{equation}
where
\[
\g_\a= 2 \pi \int_0^{\pi/2} b\left( \cos \theta \right)
             \left(\cos^\a\frac\theta2+\sin^\a\frac\theta2\right) \sin \theta d \theta\,.
\]
\end{lemm}

\begin{proof}
Since
\begin{align*}
|\cG(\vp)(\xi)-\cG(\tilde \vp)(\xi)| &=\left|\int_{\bS^2} b \left(\frac{\xi\cdot\s}{|\xi|}\right)
 [(\vp^+-\tilde \vp^+)\vp^-+\tilde \vp^+(\vp^--\tilde \vp^-)]d\s\right| \\
 &\leq\int_{\bS^2} b \left(\frac{\xi\cdot\s}{|\xi|}\right)(|\vp^+-\tilde \vp^+|
+|\vp^--\tilde \vp^-|)d\s,\\
\end{align*}
we have
\begin{align*}
\int_{\bR^3}\frac{|\cG(\vp)(\xi)-\cG(\tilde \vp)(\xi)|}{|\xi|^{3+\a}}d\xi &\leq\int_{\bR^3}\int_{\bS^2} b \left(\frac{\xi\cdot\s}{|\xi|}\right)
 \frac{|\vp^+-\tilde \vp^+|}{|\xi|^{3+\a}}d\s d\xi\\
 &\qquad + \int_{\bR^3}\int_{\bS^2} b \left(\frac{\xi\cdot\s}{|\xi|}\right)
 \frac{|\vp^--\tilde \vp^-|}{|\xi|^{3+\a}}d\s d\xi\\
&=I_1+I_2.
\end{align*}
Using the polar coordinate, we have  
\begin{align*}
  I_1 & = \int_{\bR^3}\int_{\bS^2} b\left(\frac{\xi\cdot\s}{|\xi|}\right)
        \frac{|\vp^+-\tilde \vp^+|}{|\xi|^{3+\a}}d\s d\xi\\
    & = \int_0^\infty\int_{\o\in\bS^2}\int_{\s\in\bS^2} b(\o\cdot\s)
        \frac{|\vp(r\frac{\o+\s}2)-\tilde \vp(r\frac{\o+\s}2)|}{r^{3+\a}}~r^2drd\o d\s \\
& = \int_0^\infty\iint_{\tilde\o\in\bS^2,\s\in\bS^2 \atop <\tilde\o,\s>\leq\frac\pi4}
        b(\cos(2\tilde\theta))\frac{|\vp(r\cos\tilde\theta~\tilde\o)-\tilde \vp(r\cos\tilde\theta~\tilde\o)|}{r^{3+\a}}
        ~r^2 4\cos\tilde\theta drd\tilde\o d\s,
\end{align*}
where $\tilde\o=(\o+\s)/|\o+\s|, \theta=<\o \cdot \s>,  \tilde\theta=<\tilde\o \cdot \s>$.  Here $<\mathbf{a},\mathbf{b}>$ stands for the angle
between vectors $\mathbf{a}$ and $\vb$.
Select a new variable $\tilde\s$, such that $|\tilde\s|=1,\tilde\o+\tilde\s=
        \s|\tilde\o+\tilde\s|$ and notice that $\theta=<\tilde\o,\tilde\s>$. Then, letting $\tilde r=r\cos\frac\theta2$, we have
\begin{align*}
    I_1 & = \int_0^\infty\int_{\tilde\o\in\bS^2}\int_{\tilde\s\in\bS^2} b(\cos\theta)
        \frac{|\vp(r\cos\frac\theta2~\tilde\o)-\tilde \vp(r\cos\frac\theta2~\tilde\o)|}{r^{3+\a}}
        \enskip r^2drd\tilde\o d\tilde\s\\
   & =\int_0^\infty\int_{\tilde\o\in\bS^2}\int_{\tilde\s\in\bS^2} b(\cos\theta)
        \frac{|\vp(\tilde r\tilde\o)-\tilde \vp(\tilde r\tilde\o)|}{\tilde r^{3+\a}}
        \enskip\cos^\a\frac\theta2 \enskip\tilde r^2d\tilde r d\tilde\o d\tilde\s\\
    & =\int_{\bR^3}\int_{\s\in\bS^2} b \left(\frac{\xi\cdot\s}{|\xi|}\right)
        \frac{|\vp(\xi)-\tilde \vp(\xi)|}{|\xi|^{3+\a}}
        \enskip\cos^\a\frac\theta2 d\xi d\s.\\
\end{align*}
Similarly, one can obtain
$$
I_2=\int_{\bR^3}\int_{\s\in\bS^2} b\left(\frac{\xi\cdot\s}{|\xi|}\right)
        \frac{|\vp(\xi)-\tilde \vp(\xi)|}{|\xi|^{3+\a}}
        \enskip\sin^\a\frac\theta2 d\xi d\s.
$$
As a result,
\begin{align*}
I_1+I_2&=\int_{\bR^3}\int_{\s\in\bS^2}b \left(\frac{\xi\cdot\s}{|\xi|}\right)
      \frac{|\vp(\xi)-\tilde \vp(\xi)|}{|\xi|^{3+\a}}
      \enskip\left(\cos^\a\frac\theta2 +\sin^\a\frac\theta2\right)d\xi d\s\\
    &=\g_\a||\vp-\tilde \vp||_{\cM^\a}.
\end{align*}
And this completes the proof of the lemma.
\end{proof}

We are now ready to prove Proposition 3.1.

\begin{proof}[Proof of Proposition \ref{cutoff-existence}]
The solution to \eqref{c-p-fourier}  can be obtained as a fixed point of \eqref{integral-eq} via the Banach contraction principle to
the nonlinear operator 
$$
\cF(\vp)(t,\xi)\equiv\vp_0(\xi)e^{-\g_2t}+\int_0^te^{-\g_2(t-\tau)}\cG(\vp(\tau))(\xi)d\tau,
$$
for a fixed $\varphi_0 \in \cM^\a$.
To show this, for $T >0$, denote $
\tilde \cX_T^\a=C([0,T],\cM^\a)
$
supplemented with the metric 
\begin{align*}
||\vp-\p||_{\tilde \cX_T^\a}&=\sup_{\tau \in[0,T]}   dis_{\a,\a}( \vp(\tau), \p(\tau)) , \\
\Big(&= \sup_{\tau \in[0,T]}\Big(||\vp(\tau)-\p(\tau)||_{\cM^\a} +||\vp(\tau)-\p(\tau)||_\a\Big) \Big),
\end{align*}
for  $\vp,\p\in \tilde \cX_T^\a$. 
By Lemma 4.6 of \cite{Cannone-Karch-1} and Lemma \ref{lemm01}, we can obtain
\begin{eqnarray}
  ||\cF(\vp)-1||_{\tilde \cX_T^\a} &\leq& dis_{\a,\a}(\vp_0, 1) +\g_\a T||\vp-1||_{\tilde \cX_T^\a}\,,\\
  ||\cF(\vp)-\cF(\p)||_{\tilde \cX_T^\a} &\leq& \g_\a T||\vp-\p||_{\tilde \cX_T^\a}\quad 
  \forall \vp,\p\in \tilde \cX_T^\a.
\end{eqnarray}
Hence, $ \cF: \tilde \cX_T^\a\to \tilde \cX_T^\a$ is a contraction mapping, if we choose $T$ such that $\g_\a T<1.$ Then Banach contraction principle 
implies that there exists a unique solution on $[0,T]$ where $T$ is independent of the initial condition. Consequently, repeating this procedure, we can construct the unique solution on any finite time interval.

It follows from \eqref{c-p-fourier} that the function
\[
H(t,\xi) = \frac{\varphi(t,\xi) -\tilde \vp(t,\xi)}{|\xi|^{3+\a}}
\]
satisfies
\[
e^{\g_2 t} H(t,\xi) = H(0,\xi) + \int_0^t e^{\g_2 \tau}\frac{\cG(\vp(\tau))(\xi) -\cG(\tilde \vp(\tau))(\xi)}{|\xi|^{3+\a}} d\tau.
\]
Integrating with respect to $\xi$ over $\RR^3$, by Lemma \ref{lemm01} we have
\[
e^{\g_2 t} ||\vp(t)-\tilde \vp(t)||_{\cM^\a} \leq||\vp_0-\tilde \vp_0||_{\cM^\a} +\g_\a  \int_0^t e^{\g_2 \tau}
||\vp(\tau)-\tilde \vp(\tau)||_{\cM^\a} d\tau\,,
\]
which gives \eqref{cutoff-sta}. And this completes the proof of the
proposition.
\end{proof}

\subsection{Stability and Existence of Solutions in non-cutoff case}\label{3-2}
In this subsection, we 
assume that $b$ satisfies \eqref{index-sing} for some $\a_0 \in (0,2)$ and let  $\a \in [\a_0, 2)$. 
Let $b_n(\cos \theta) = \min \{b(\cos \theta), n\}$.  By Proposition \ref{cutoff-existence}, for any $\varphi_0 \in \cM^\a$ we have a unique solution  
$\vp_n(t,\xi) \in C([0,\infty); \cM^\alpha)$ to the cutoff Cauchy problem
\eqref{c-p-fourier} with $b$ replaced by $b_n$. If $\lambda_{\a, n}$ is defined by \eqref{05} with $b$ replaced by $b_n$, then it is obvious 
$\lambda_{\a, n} \le \lambda_\a$. Hence we have 
\begin{align}\label{M-estimate}
\|\vp_n (t) - 1 \|_{\cM^\alpha} \le e^{\lambda_\alpha t} \|\vp_0 - 1 \|_{\cM^\alpha}\,.
\end{align}
According to \cite{morimoto-12} ( an improvement of  \cite{Cannone-Karch-1}), it is proved that the unique solution $\varphi(t,\xi) \in C([0,\infty), \cK^\a)$ to
\eqref{c-p-fourier} for the initial datum $\varphi_0 \in \cK^\a$ can
be obtained as a limit of  a subsequence of $\{\varphi_n(t)\}_{n=1}^\infty$, by means of 
the Ascoli-Arzel\`a theorem.  
Namely, writing the subsequence $\{\varphi_n(t)\}_{n=1}^\infty$ again,
for any compact set $K \subset [0,\infty) \times \RR^3_\xi$,  we have
\begin{align}\label{uniform-convergence}
\lim_{n \rightarrow \infty} \sup_K \Big|\vp_ n(t,\xi) - \vp(t,\xi)\Big| =0 .
\end{align}
It follows from \eqref{M-estimate} that if $\varphi_0 \in \cM^\a$, then for any $0< \delta <1$ we have 
\[
\int_{\delta < |\xi| \le \delta^{-1}}  \frac{|\vp(t, \xi) - 1|}{|\xi|^{3+\alpha}} d \xi  =
\lim_{n \rightarrow \infty} 
\int_{\delta < |\xi| \le \delta^{-1}}  \frac{|\vp_n (t, \xi) - 1|}{|\xi|^{3+\alpha}} d \xi 
\le e^{\lambda_\alpha t} \|\vp_0 - 1 \|_{\cM_\alpha}\,.
\]
Letting $\delta \rightarrow 0$, we obtain $\vp(t,\xi) \in \cM^\a$ for each $t>0$.  By the same limiting procedure for \eqref{cutoff-sta},
we can show 
the stability estimate \eqref{04} in the non-cutoff case.

To complete the proof of Theorem  \ref{fourier-space},
we shall show $\vp(t,\xi) \in C([0,\infty), \cM^\a)$, that is, \eqref{fourier-continuity}. In view of two stability estimates 
\eqref{04}, \eqref{alpha-stability}, it suffices to show 
\begin{prop}\label{continuity-time-in-M}
Let $\varphi(t,\xi) $ be the unique solution to
\eqref{c-p-fourier} for the initial datum $\varphi_0 \in \cM^\a$. Then 
for any $0 \le s < t $, we have 
\begin{align}\label{1st}
&\|\vp(t) - \vp(s)\|_{\a} \lesssim 
|t-s| \sup_{s \le \tau \le t} \|1-\vp(\tau) \|_{\alpha}\,, \\
&\|\vp(t) - \vp(s)\|_{\cM^\a} \lesssim 
|t-s| \sup_{s \le \tau \le t} \|1-\vp(\tau) \|_{\cM^\alpha}\, . \label{second}
\end{align}

\end{prop}

The first estimate is a direct consequence of 
the formula 
\begin{align*}
\vp(t,\xi) -\vp(s,\xi) 
= \int_s^t \int_{\SS^2}b\left(\frac{\xi \cdot \sigma}{|\xi|}\right) \Big( \vp(\tau, \xi^+)\vp (\tau, \xi^-) - \vp (\tau, \xi)
\Big) d\sigma d \tau,
\end{align*}
and Lemma 2.2 of \cite{morimoto-12}. The second one follows from 
the following lemma which is a variant of Lemma 2.2 of \cite{morimoto-12}; 

\begin{lemm}\label{main-M-lemma} 
If $\varphi \in \cM^\alpha$ for $\alpha \in [\alpha_0,2)$, then 
\begin{align}\label{funda}
&\int_{\RR^3_\xi} \frac{1}{|\xi|^{3+\alpha}} \left|\int_{\SS^2}b\left(\frac{\xi \cdot \sigma}{|\xi|}\right) \Big( \varphi(\xi^+)\varphi(\xi^-) - \varphi(\xi)
\Big) d\sigma\right| d\xi\\
&\qquad \qquad 
\lesssim \left(\int_0^1 (1-\tau)^{\alpha/2} b(\tau) d\tau\right)
\|1-\varphi\|_{\cM^\alpha}\,. \notag
\end{align}
\end{lemm}
\begin{proof} As in \cite{morimoto-12}, we put  $\zeta = \left(\xi^+ \cdot \frac{\xi}{|\xi|}\right) \frac{\xi}{|\xi|}$ and 
consider $\tilde \xi^+ = \zeta -(\xi^+-\zeta)$, 
which is
symmetric to $\xi^+$ on $\SS^2$, see Figure 1. 
\begin{figure}[tbh]\label{f-2-3}
\begin{center}
\includegraphics[width=2.5in]{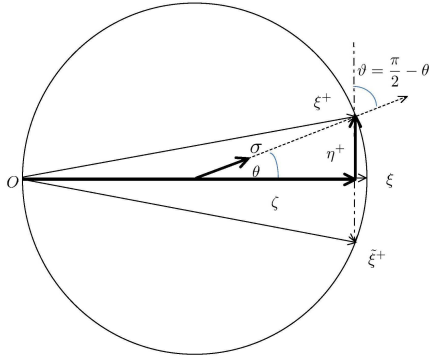}
\caption{$\theta = < \frac{\xi}{|\xi|} \cdot \sigma > $, \enskip  $\vartheta = <\frac{\eta^+}{|\eta^+|}, \sigma>$, $\eta^+ =\xi^+ -\zeta$. }
\end{center}
\end{figure}
We divide the integrand of the left hand side of \eqref{funda} into
\begin{align*}
&\int_{\SS^2}b\left(\frac{\xi \cdot \sigma}{|\xi|}\right) \Big( \varphi(\xi^+)\varphi(\xi^-) - \varphi(\xi)
\Big) d\sigma\\
&= \int_{\SS^2}b\left(\frac{\xi \cdot \sigma}{|\xi|}\right) \Big( \varphi(\xi^+) - \varphi(\xi)
\Big) d\sigma + \int_{\SS^2}b\left(\frac{\xi \cdot \sigma}{|\xi|}\right) \varphi(\xi^+)
\Big(\varphi(\xi^-) -1\Big) d\sigma\\
&= \frac{1}{2}\int_{\SS^2}b\left(\frac{\xi \cdot \sigma}{|\xi|}\right) \Big( \varphi(\xi^+) + \varphi(\tilde \xi^+) - 2 \varphi(\xi)
\Big) d\sigma \\
& \qquad \qquad + \int_{\SS^2}b\left(\frac{\xi \cdot \sigma}{|\xi|}\right) \varphi(\xi^+)
\Big(\varphi(\xi^-) -1\Big) d\sigma\\
&= \frac{1}{2}\int_{\SS^2}b\left(\frac{\xi \cdot \sigma}{|\xi|}\right) \Big( \varphi(\xi^+) + \varphi(\tilde \xi^+) - 2 \varphi(\zeta) \Big)d\sigma \\
&\quad +\int_{\SS^2}b\left(\frac{\xi \cdot \sigma}{|\xi|}\right)
\Big(\varphi(\zeta) - \varphi(\xi) \Big) d\sigma
+ \int_{\SS^2}b\left(\frac{\xi \cdot \sigma}{|\xi|}\right) \varphi(\xi^+)
\Big(\varphi(\xi^-) -1\Big) d\sigma\\
& = I_1(\xi) + I_2(\xi) + I_3(\xi)\,.
\end{align*}
Putting $\eta^+ = \xi^+ - \zeta$, we have 
\begin{align*}
&|\varphi(\xi^+) + \varphi(\tilde \xi^+) - 2 \varphi(\zeta) |=\left
|\int_{\RR^3} e^{-i\zeta \cdot v}\Big( e^{-i \eta^+ \cdot v} + e^{ i \eta^+\cdot v}  - 2\Big)d F(v) \right|\\
&\qquad \le \int_{\RR^3}
|e^{-i \zeta \cdot v}| \Big( 2 -e^{-i\eta^+ \cdot v} - e^{i\eta^+ \cdot v}\Big) dF (v). 
\end{align*}
Noting  $|\eta^+| = |\xi| \cos(\theta/2)\sin(\theta/2)$ and using the change of variables
$\xi \rightarrow ( \xi^+ \rightarrow ) \eta^+$,
we have 
\begin{align*}
\int_{\RR^3_\xi}
\frac{|I_1(\xi)|}{|\xi|^{3+\alpha}} d\xi  &\lesssim 
\int_{\RR^3}\frac{|1-\varphi(\pm \eta^+)|}{
| \eta^+ |^{3+\a}}d \eta^+
 \int_0^{\pi/2}(\sin \theta)^{1+\alpha} b(\cos \theta)\cos \theta 
d \theta\\
& \lesssim \|1-\varphi \|_{\cM^\a}\,,
\end{align*}
where by the same arguments in Section 4 of \cite{MUXY-DCDS} (cf. \cite{ADVW}), we have used the fact that 
\begin{align*}
d\xi = \frac{4d\xi^+}{\cos^2(\theta/2)} \enskip, \enskip d \xi^+ = \frac{4 d\eta^+}{\sin^2 (\theta/2)}\, \enskip , \enskip
d\eta^+ d \sigma  = 2 \pi \sin \vartheta d \vartheta  d \eta^+ \enskip (\vartheta = \pi/2 - \theta)\,.
\end{align*}

Notice that 
\begin{align}\label{important-est}
|\varphi(\xi) - \varphi(\xi+\eta)| \le 4 |1-\varphi(\xi)|^{1/2}  |1-\varphi(\eta)|^{1/2}  + |1-\varphi(\eta)| \,,
\end{align}
which was proved in the proof of (19) in \cite{morimoto-12}.
By \eqref{important-est} with $\eta = \zeta -\xi$,  we have
\begin{align*}
\int_{\RR^3_\xi}
\frac{|I_2(\xi)|}{|\xi|^{3+\alpha}} d\xi & \lesssim \int_{\SS^2}b\left(\frac{\xi \cdot \sigma}{|\xi|}\right)
\int_{\RR^3}\frac{|1-\varphi(\xi)|^{1/2} |1-\varphi(\zeta-\xi)|^{1/2}}{
| \xi |^{3+\a}}d \xi d\sigma \\
&\quad + \int_{\SS^2}b\left(\frac{\xi \cdot \sigma}{|\xi|}\right)
\int_{\RR^3}\frac{|1-\varphi(\zeta-\xi)|}{
| \xi |^{3+\a}}d \xi d\sigma \\
& = J_1 + J_2 \,.
\end{align*}
Using the Schwarz inequality, we have 
\begin{align*}
J_1 & \le\int_{\SS^2}b\left(\frac{\xi \cdot \sigma}{|\xi|}\right)\Big( 
\int_{\RR^3}\frac{|1-\varphi(\xi)|}{
| \xi |^{3+\a}}d \xi  \Big)^{1/2}\Big( 
\int_{\RR^3}\frac{|1-\varphi(\zeta- \xi)|}{
| \xi |^{3+\a}}d \xi  \Big)^{1/2}d\sigma \\
&= 2\pi  \|1-\varphi \|_{\cM^\a}^{1/2} \Big( 
\int_{\RR^3}\frac{|1-\varphi(\eta)|}{
| \eta |^{3+\a}}d \eta  \Big)^{1/2}
\int_{0}^{\pi/2} b(\cos \theta)\sin^\a(\theta/2) \sin \theta  d\theta  \\
&\lesssim \|1-\varphi \|_{\cM^\a}\,,
\end{align*}
where we have used $|\eta|=|\zeta -\xi| = |\xi|\sin^2(\theta/2)$. The estimation for 
$J_2$ is easier, so that we omit its proof.
Also since  similar estimate holds for $I_3(\xi)$, we obtain the desired estimate  \eqref{funda}. And this completes the proof of the lemma.
\end{proof}

\subsection{Continuity of measure valued solutions} 
Finally in this subsection, we discuss the continuity of the solutions
obtained above.
Assume that $b$ satisfies \eqref{index-sing} for some $\a_0 \in (0,2)$ and let  $\a \in [\a_0, 2), \a \ne 1$. 
If $F_0 \in P_\a(\RR^3)$, then it follows from \eqref{equivalence-space} that $\vp_0 = \cF(F_0)$ belongs to $\cM^\a$.
By Theorem \ref{fourier-space} and \eqref{equivalence-space}, 
there exists a unique measure valued solution $F_t \in C([0,\infty), P_\a(\RR^3))$ to the Cauchy problem 
\eqref{bol}-\eqref{initial}. Here the continuity with respect to $t$ 
is in the following sense:
\begin{align*}
&\lim_{t \rightarrow t_0} \int \psi(v) dF_t(v) = \int \psi(v) dF_{t_0}(v)  
\mbox{ for any  $\psi \in C(\RR^d)$} \\
&\mbox{\qquad  \qquad \qquad 
 satisfying 
the growth condition $|\psi(v)| \lesssim \la v \ra^\a$.} 
\end{align*}
This continuity follows from the  same argument as in the last paragraph of Section \ref{s2}. 
Indeed, the fact that $\cF(F_t) =\varphi(t,\xi) \in C([0,\infty), \cM^\a)$ yields  
\begin{align*}
\int_{\{|\xi| \le 1/R\}} \frac{|1-\vp(t, \xi)|}{|\xi|^{3+\alpha}} d \xi \le \int_{\{|\xi| \le 1/R\}} \frac{|1-\vp(t_0, \xi)|}{|\xi|^{3+\alpha}} d \xi
+ \|\vp(t) - \vp(t_0)\|_{\cM^\a},
\end{align*}
and hence for any $\varepsilon >0$, there exist  $\delta >0$ and $R >0$ such that 
\[
\sup_{t \in [t_0 -\delta, t_0 +\delta]} \int_{\{|\xi| \le 1/R\}} \frac{|1-\vp(t, \xi)|}{|\xi|^{3+\alpha}} d \xi < \varepsilon.
\]
This and \eqref{uniform-moment-est} imply that 
\[
\sup_{t \in [t_0 -\delta, t_0 +\delta]} \int_{\{|v| \ge R\}} |v|^\alpha dF_t(v) < \varepsilon / 2c_{\alpha,3,1}.
\]
Since $
F_t \in C([0,\infty), P)$,  we  have $
F_t \in C([0,\infty), P_\a)$. {Thus Theorem \ref{existence-base-space} is now proved.
If $F_0 \in P_1(\RR^3)$ satisfies $
 \int_{\RR^3} \la v \ra \log \la v \ra dF_0(v) < \infty
$, then it follows from Proposition \ref{chara-1} that  $\vp_0 =\cF(F_0)$ belongs to $\cM^1$, and hence the last statement of Remark  \ref{remark-1} is
also obvious. On the other hand, if we 
simply assume $F_0 \in P_1(\RR^3)$, noting that the index  $\a_0$ of the condition
\eqref{index-sing} is in $(0,1]$,  we have a weaker result, \eqref{1-mome} in Remark  \ref{remark-1}, by means of the following proposition.
}
\begin{prop}\label{propagation-moments}
Let $\alpha \ge 1$ and let $b$ satisfy \eqref{index-sing} with $\a_0 =1$.  If the initial data $F_0 \in P_\alpha$, then 
the unique solution $F_t$ belongs to $ P_\alpha$ for any $t>0$. More 
precisely, there exists a constant $C >0$ independent of $t$ such that
\begin{align}\label{prop-moment}
\int \la v \ra^\alpha dF_t(v) \le C e^{Ct} \int\la v \ra^\alpha  dF_0(v)\,.
\end{align}
\end{prop}
\begin{proof}
For the simplicity of the notations, we consider the case where $F_0, F_t$ have  density functions $f_0(v), f(t,v)$. 
For $\delta >0$, put 
\[W_\delta(v) = \frac{\la v \ra^\alpha}{1+\delta \la v\ra^\alpha}\,.
\]
Since $x/(1+\delta x)$ is increasing in $[1,\infty]$ and $|v'| \le |v| +|v_*|$,  we have
\begin{align*}
W_{\delta}(v') &\lesssim \frac{\la v\ra^\alpha + \la v_*\ra^\alpha}{1+\delta (\la v\ra^\alpha + \la v_*\ra^\alpha)}
= \frac{\la v\ra^\alpha }{1+\delta (\la v\ra^\alpha + \la v_*\ra^\alpha)}
+\frac{ \la v_*\ra^\alpha}{1+\delta (\la v\ra^\alpha + \la v_*\ra^\alpha)}\\
&\le \frac{\la v\ra^\alpha }{1+\delta \la v\ra^\alpha }
+\frac{ \la v_*\ra^\alpha}{1+\delta \la v_*\ra^\alpha} = W_{\delta}(v)+W_{\delta}(v_*)\,.
\end{align*}
Therefore,
$|W_{\delta}(v') - W_{\delta}(v)| \lesssim W_{\delta}(v)+W_{\delta}(v_*) $. If $b_n(\cos \theta) = \min \{b(\cos \theta), n \} $ and $f^n(t,v)$ is the
unique solution of the corresponding Cauchy problem, then 
we have
\begin{align*}
\frac{d }{dt} \int f^n(t,v)W_{\delta}(v) dv \lesssim 
\Big(\int b_n d\sigma \Big) \Big(\int f^n(t,v)W_{\delta}(v) dv \Big) \Big(\int f^n(t,v_*)dv_* \Big)\,,
\end{align*}
which yields 
\[
\int f^n(t,v)W_{\delta}(v) dv \le C_n e^{C_n t} \int f_0(v)W_{\delta}(v) dv.
\]
Taking the limit $\delta \rightarrow +0$, we have $f^n(t,v) \in L^1_\alpha$. 
Note that
\[
\la v' \ra^\alpha - \la v \ra^\alpha \lesssim \big(\la v  \ra^{\alpha-1} +  \la v_*  \ra^{\alpha-1}\big)
|v-v_*|\sin \frac{\theta}{2},\]
because $\a \ge 1$. 
Then, there exists a constant $C >0$ independent of $n$ such that 
\begin{align*}
\frac{d }{dt} \int f^n(t,v)\la v \ra^\alpha  dv &\le 
\Big(\int b_n \theta d\sigma \Big) \Big(\int f^n(t,v) \la v \ra^\alpha dv \Big) \Big(\int f^n(t,v_*)dv_* \Big)\\
&\le C \int f^n(t,v) \la v \ra^\alpha dv\,.
\end{align*}
We have 
$\int f^n(t,v)\la v \ra^\alpha  dv \le C e^{Ct} \int f_0(v)\la v \ra^\alpha  dv$.
Take a cutoff function $\chi(v)$ in $C_0^\infty (\RR^3)$ satisfying 
$0\le \chi \le 1$ and  
$\chi = 1$ on $\{|v| \le 1\}$.
Then, for any $m \in \NN$,   we have 
$$\int f^n(t,v)\la v \ra^\alpha\chi\left(\frac{v}{m}\right)   dv \le C e^{Ct} \int f_0(v)\la v \ra^\alpha  dv. $$
Since $f^n(t,v)\rightarrow f(t,v)$ in $\cS'(\RR^3_v)$ and $\la v \ra^\alpha \chi\left(\frac{v}{m}\right)  \in \cS$, we get
 $$\int f(t,v)\la v \ra^\alpha \chi\left(\frac{v}{m}\right)     dv \le C e^{Ct} \int f_0(v)\la v \ra^\alpha  dv. $$
 Letting $m \rightarrow  \infty$ gives the desired estimate \eqref{prop-moment}. 
And this completes the proof of the proposition.
 \end{proof}

\section{Proof of Theorem \ref{smoothing-thm}}\label{s3}
\setcounter{equation}{0}

In this section, we will show that the measure valued solutions obtained
in both the function spaces $\cK^\alpha$ and $\cM^\alpha$ is $H^\infty_v$ for
any positive time, { under the  angular singularity assumption \eqref{1.2} on the cross section.}

For this, we first recall that in the proof of Theorem 1.3 of \cite{MY}, we 
already showed that, if $F_0 \in \tilde P_\a(\RR^3)$, then there exists a $T>0$ such that
the unique solution $f(t,v) \in  C([0,\infty), \tilde P_{\alpha})$ ($\alpha >2s$)  satisfies 
\[
f(t,v) \in H^\infty(\RR^3) \enskip \mbox{for any fixed} \enskip  t \in (0,T]\,.
\]

This local smoothing effect was extended to the global one in the case when the initial datum belongs to $P_2$,  by using the energy 
conservation law and the uniform boundedness of the entropy norm (see (1.12) of \cite{MY}). 
 
In order to prove the global in time smoothing effect even for infinite
energy solutions, instead of (1.12) in \cite{MY}, we will show that for any $T_1 > T$, there exists a $C_{T_1} >0$ such that 
\begin{align}\label{semi-global-uniform}
\sup_{T \le t \le T_1} \Big(\|f(t)\|_{L^1_{\alpha'} }+ \|f(t)\|_{L \log L}\Big) < C_{T_1},
\end{align}
for any $\alpha' < \alpha$.   Based on  \eqref{semi-global-uniform},
 we will then show the global in time smoothing effect, that is, $f(t,v) \in H^\infty(\RR^3)$ for any 
$t >T$.
 
The bound of the first term in \eqref{semi-global-uniform} is now clear. In fact, 
by using \eqref{moment-es} and \eqref{trivial-estimate},
we have 
\begin{align}\label{general-mome}
\||v|^{\a'}f(t)\|_{L^1 } \le    \frac{C_{\a',\a,3}}{2c_{\alpha',3, \infty}}\big(\|\varphi(t)-1\|_\a\big)^{\a'/\a} \le
 \frac{C_{\a',\a,3}}{2c_{\alpha',3, \infty}}\big(e^{\lambda_\a t}\|\varphi_0-1\|_\a \big)^{\a'/\a}\,,
\end{align}
where we have used \eqref{alpha-stability} to get the second inequality.
%
It remains to show the boundedness of the second term in
\eqref{semi-global-uniform}. 

It follows from the local smoothing effect that
$\int f(T,v) \log (1+ f(T,v) )dv < \infty$. Writing $0$ and $\a$ instead of $T$ and $\alpha'$, respectively,
for simplicity,  we show the following;

\begin{prop}\label{entropy-finite}
Let $b(\cos \theta)$ satisfy \eqref{1.2} with $0<s<1$. Let $2s< 
\alpha<2$. 
Assume that $0 \le f_0 \in L^1_{\alpha}(\RR^3)  \cap  L \log L(\RR^3)$ and $\int f_0(v) dv =1$. 
If $f(t,v)$ is the unique solution in $C([0,\infty), \tilde P_{\alpha})$   of the Cauchy problem  \eqref{bol}-\eqref{initial}, then $f(t,v) \in L^\infty_{loc}([0,\infty), L^1_{\alpha}(\RR^3)  \cap  L \log L(\RR^3))$, that is, 
for any $T_1 >0$ there exists a $C_{T_1} >0$ such that 
\begin{align}\label{boundedness-234}
\sup_{0< t \le T_1} \int \la v \ra^{\alpha} f(t,v)   dv + \int f(t,v) \log \big( 1 +  f(t,v) \big) dv \le C_{T_1}\,.
\end{align}
\end{prop}
\begin{proof}
Since we  assume $f_0 \in L^1_\a(\RR^3)$, there exists $C >0$ such that 
\begin{equation}\label{mome-again-1}
 \int \la v \ra^{\alpha} f(t,v)   dv \le C e^{Ct}\int \la v \ra^{\alpha} f_0(v)   dv.
\end{equation}
In fact, when $\a \ne 1$, 
it follows from \eqref{moment-es}, \eqref{04} and \eqref{fourier-int} that
\begin{align}\label{better-mome}
\||v|^\a f(t)\|_{L^1 } \le \frac{1}{2 c_{\a,3,\infty}}\|\varphi(t)-1\|_{\cM^\a} \le   \frac{ e^{\lambda_\a t}}{2 c_{\a,3,\infty}} \|\varphi_0-1\|_{\cM^\a}
\lesssim e^{\lambda_\a t} \|f_0\|_{L^1_{\alpha} } \,.
\end{align}
The exceptional case $\a =1$ follows from
Proposition \ref{propagation-moments}, 
because of $2s < \a =1$.
We will show the boundedness  of the second term in \eqref{boundedness-234}.
Take the same cutoff function $\chi(v)$ in $C_0^\infty(\RR^3)$ as in the proof of Proposition \ref{propagation-moments}.
For $m \in \NN$, put $\tilde f_{0,m}(v) = c_m \chi\left(\frac{v}{m}\right) f_0(v)$ with  $c_m = 1/ \int \chi\left(\frac{v}{m}\right)f_0(v)dv$.
Since 
$f_0(v) \in L^1_{\alpha}(\RR^3)$, there exists $M \in \NN$
such that  $1 \le  c_m \le 2$ for all $ m \ge M$.  

We consider only the case $s \ge 1/2$ because the case when $0<s<1/2$ is easier.
For 
$a_m = \int v \tilde f_{0,m}(v) dv$,  put $f_{0,m}(v) = \tilde f_{0,m}(v+a_m)$. 
Since $|a_m|^\alpha \le 2 \int \la v \ra^\alpha f_0 dv$,  it follows from Lemma 3.15  of \cite{Cannone-Karch-1}
 that $\{f_{0,m}\}$ belongs to a bounded set of $\tilde P_\alpha$,
equivalently, $\{\hat f_{0,m}\}$ belongs to a bounded set of $\cK^\alpha$,  that is,
\begin{align}\label{3-13}
\|1- \hat f_{0,m}\|_\alpha \lesssim \int \la v \ra^{\alpha} f_0(v) dv \,.
\end{align}
Consider the  solutions $f_m(t,v)$ for the Cauchy problem 
with the initial data $f_{0,m}(v)$ with $m \ge M$. Since $f_{0,m} \in L^1_2(\RR^3) \cap L \log L(\RR^3)$, we have 
by Theorem 1 of \cite{villani} that 
\begin{align}\label{H-theorem}
\int f_m(t,v) \log f_m(t,v) dv  + \int_0^t D(f_m)(s) ds =  \int f_{0,m}(v) \log f_{0,m}(v) dv\,, 
\end{align}
where $D(f)(t)$ is defined by 
\[
D(f) (t) = \frac{1}{4}\iiint_{\RR^3 \times \RR^3 \times \SS^2} b(f'_*f' -f_* f) \log \frac{f'_*f'}{ f_* f}dvdv_*d\sigma \ge 0,
\]
with $f = f(t,\cdot)$. Therefore, writing $\chi_m = \chi(v/m)$, we have 
\begin{align*}
&\int f_m(t,v) \log f_m(t,v) dv  \leq \int f_{0,m}(v) \log f_{0,m}(v) dv\\
&\le 2\log 2 \int f_0 dv  + 2  \int \chi_m f_0 \log^+  \chi_m f_0  dv.
\end{align*}
Noting  that $x \log x \ge -y + x \log y$ $(x \ge 0, y >0)$, we have
\begin{align*}
0 \ge f_m(t,v) \log^-f_{m}(t,v) \ge - e^{-\la v \ra^{\alpha}} - \la v \ra^{\alpha}   f_{m}(t,v),
\end{align*}
which together with \eqref{mome-again-1} yields,
\begin{align*}
&
\int f_m(t,v) \log^+ f_m(t,v) dv  
\le 2 \log 2 \int f_0 dv \\
& \quad + 2 \int f_{0}(v) \log^+  f_{0}(v) dv  +\int e^{-\la v \ra^{\alpha} } dv  +  4C e^{Ct} \int  \la v \ra^\alpha f_{0}(v) dv, 
\end{align*} 
because $\int \la v\ra^\a f_{0,m} dv \le 4 \int \la v\ra^\a f_{0} dv$.
Thus, for any $ T>0$ there exists $C'_T >0$ such that
\begin{align*}
\sup_{0 < t \le T}\Big( \int  \la v \ra^{\alpha}  f_m(t,v) dv+ \int f_m(t,v) \log (1+ f_m(t,v) )dv \Big)  \le C'_T\,,
\end{align*}
which concludes the weak compactness of $\{ f_m \}$ in $L^1(\RR^3)$, by means of Dunford-Pettis criterion.
Notice  \eqref{3-13} again,  namely the fact that  $\hat f_{0,m}, \hat f_0 \in \cK^{\alpha}$ uniformly with respect to $m$. 
Take a $\a' \in (2s,\a)$.  For any $\delta >0$  we have 
\begin{align*}
\sup_{0<|\xi| < \delta} \frac{|\hat f_{0,m} (\xi) - \hat f_{0}(\xi)|}{|\xi|^{\alpha'}} &\leq  \delta^{\alpha-\alpha'} \Big( \sup_{0<|\xi| < \delta} \frac{|\hat f_{0,m} (\xi) - 1|}{|\xi|^{\alpha}}
+ \sup_{0< |\xi| < \delta} \frac{|1 - \hat f_{0}(\xi)|}{|\xi|^\alpha}\Big)\\
&\lesssim  \delta^{\alpha - \alpha'} \int \la v \ra^{\alpha} f_0(v)dv\,.
\end{align*}
Note that for any  $R >0$ large enough, we have
\begin{align*}
\sup_{|\xi| >  R} \frac{|\hat f_{0,m} (\xi) - \hat f_{0}(\xi)|}{|\xi|^{\alpha'}} \le2 R^{-\alpha'} \,.
\end{align*}
In view of \eqref{3-13},  it follows from Lemma 2.1 of \cite{morimoto-12}  that $\{\hat f_{0,m}\}$ is uniformly equi-continuous on the compact set $\{\delta \le |\xi| \le R\}$.
By the Ascoli-Arzel\'a theorem, $\{\hat f_{0,m}\}$ is a Cauchy sequence in $\|\cdot \|_{\alpha'}$, taking a subsequence if necessary.  Since $f_{0,m}(v) \rightarrow f_0(v)$ in $\cS'(\RR^3)$ as $m$ tends to  $\infty$,
we concludes $\|\hat f_{0,m}-  \hat f_{0}\|_{\alpha'} \rightarrow 0$.

It follows from  \eqref{alpha-stability}  that
\[
\|\hat f_m(t,\cdot) - \hat f(t,\cdot)\|_{\alpha'} \le e^{\lambda_{\alpha'} t}\|\hat f_{0,m} - \hat f_0\|_{\alpha'},
\]
which implies that $\hat f_m(t,\xi) \rightarrow \hat f(t,\xi)$ everywhere in $\xi$ for any fixed $t>0$.
Since 
$$|\hat f_m(t,\xi)|,\quad |\hat f(t,\xi)| \le 1,$$
we know 
$f_m(t,v) \rightarrow f(t, v)$ in $\cS'(\RR^3)$.
If we recall the weak compactness of $\{f_m(t,\cdot)\}$ in $L^1(\RR^3)$, then we obtain 
\[
f_m(t,\cdot)  \rightarrow f(t,\cdot) \enskip \mbox{weekly in } \enskip L^1(\RR^3).
\]
Since $\lambda \log(1+ \lambda)$ is convex, for $0<t\le T$, we have 
\[
\int f(t,v) \log (1+f(t,v)) dv \le \lim \inf \int f_m(t,v) \log (1+f_m(t,v))dv  \le C'_T.
\]
And this completes the proof of the proposition.
\end{proof}

Since the global in time smoothing effect has been established, that is, $f(t,v) \in H^\infty(\RR^3)$ for any $t > 0$, 
it follows from Theorem \ref{existence-base-space} together with Remark \ref{remark-1} that
$f(t,v) \in L^1_\a(\RR^3)$ for $t >0$ if $F_0 \in P_\a(\RR^3)$. 
To complete the proof of Theorem \ref{smoothing-thm} , it remains to show $f(t) \in$ $ C((0, \infty), L^1_\a(\RR^3)\cap H^\infty(\RR^3))$ if 
$F_0 \in P_\a(\RR^3)$ and $\a \ne 1$.  
Let  $0 < T_1 < T_2$,
Since it follows from \eqref{second} that
\[
\|\varphi(t) - \varphi(t_0)\|_{\cM^\a} \lesssim |t-t_0|\, , \enskip t, t_0 \in [T_1, T_2],
\]
by the same argument as the one given in the last paragraph of Section \ref{s2},
we see that
 for any $\varepsilon >0$ there exist $R >1$ and $\delta >0$ such that 
\begin{align}\label{uniform-mome-decay}
\int_{\{|v| \ge R \}} |v|^\alpha f(t,v)dv 
< \varepsilon\, \enskip \mbox{if $|t-t_0| <\delta$}.
\end{align}
Then  for any $M >1$,  we have 
\begin{align*}
 |f(t,v) -f(t_0,v)|  &\le \int |\varphi(t,\xi)-\varphi(t_0,\xi)|d\xi \\
& \le  2\sup_{\tau \in [T_1, T_2]}\|f(\tau)\|_{H^2(\RR^3)} \Big(\int_{\{|\xi| \ge M\}} \la \xi \ra^{-4} d\xi \Big)^{1/2} \\
& \quad + \|\varphi(t) -\varphi(t_0)\|_\a \int_{\{|\xi| < M\}} |\xi|^{\a} d\xi\,, 
\end{align*}
which together with \eqref{uniform-mome-decay} imply that
 $f(t) \in C((0,\infty), L^1_\a(\RR^3))$. 
In view of \eqref{semi-global-uniform}, it follows from the proof of Theorem 1.5 in \cite{morimoto-12} that
for any $N \in \NN$ there exists a $ C_N >0$ such that
\[
\sup_{T_1 \le \tau \le T_2 } \int |\la \xi \ra^{N+1} \varphi(\tau, \xi)|^2 d \xi \le C_N.
\]
Noticing that for any $R >1$
\begin{align*}
\|f(t) - f(t_0) \|^2 _{H^N} \le  (1+ R^2)^N \int_{\{|\xi| \le R\} }|\varphi(t,\xi) -\vp(t_0,\xi)|^2 d \xi + 4 C_N R^{-2},
\end{align*}
we have $f(t) \in  C((0, \infty),  H^N(\RR^3))$ because of $\vp(t) \in C([0, \infty), \cK^\a)$.
\bigskip


\bigskip
\noindent
{\bf Acknowledgements:} The research of the first author was supported in part 
by  Grant-in-Aid for Scientific Research No.25400160,
Japan Society of the Promotion of Science. The research of the third author was
supported in part by the General Research Fund of Hong Kong, CityU No.104511,
and the Shanghai Jiao Tong University.


\end{document}